\documentclass[10pt]{amsart}

\usepackage{amsmath}
\usepackage{amssymb,amsfonts}
\usepackage{amsthm}
\usepackage{enumitem}
\usepackage{tikz}

\usepackage[cmtip,arrow]{xy}
\usepackage{pb-diagram, pb-xy}

\theoremstyle{plain}
\newtheorem{definition}{Definition}[section]
\newtheorem{theorem}[definition]{Theorem}
\newtheorem{corollary}[definition]{Corollary}
\newtheorem{lemma}[definition]{Lemma}
\newtheorem{proposition}[definition]{Proposition}

\newtheorem{remark}[definition]{Remark}
\newtheorem*{theorem*}{Theorem}

\newcommand{\Z}{\mathbb{Z}}
\newcommand{\Q}{\mathbb{Q}}

\newcommand{\Hom}{{\rm Hom}}
\newcommand{\Ext}{{\rm Ext}}
\newcommand{\Tor}{{\rm Tor}}
\newcommand{\Rmod}{R\text{-}{\rm Mod}}

\newcommand{\FP}[1]{\rm{FP}_{#1}}
\newcommand{\FPinj}[1]{\rm{FP}_{#1}\text{-}{\rm Inj}}
\newcommand{\FPflat}[1]{\rm{FP}_{#1}\text{-}{\rm Flat}}

\newcommand{\FPid}[1]{\rm{FPdim}({#1})}

\newcommand{\kker}[1]{\ensuremath{{{\rm Ker}\left(#1\right)}}}

\begin{document}
\title[$\FP{n}$-injective/flat and Gorenstein AC-flat covers and preenvelopes.]{$\FP{n}$-injective and $\FP{n}$-flat covers and preenvelopes, and Gorenstein AC-flat covers.}

\author{Daniel Bravo}
\address{
 Instituto de Ciencias F\'isicas y Matem\'aticas,
 Facultad de Ciencias,
 Universidad Austral de Chile,
 Valdivia, Chile}
 \email{daniel.bravo@uach.cl}

\author{Sergio Estrada}
\address{
Departamento de Matem\'aticas,
Universidad de Murcia,
Campus de Espinardo,
30100 Murcia, Spain
}
 \email{sestrada@um.es}

\author{ Alina Iacob}
\address{
Department of Mathematical Sciences,
Georgia Southern University,
Statesboro, GA 30460, U.S.A
}
 \email{aiacob@georgiasouthern.edu}

\date{\today}



\thanks{This research has started while the authors visited the Banff International Research Station for Mathematical Innovation and Discovery (BIRS), under the ``Research in Teams" program in May 2016 (http://www.birs.ca/events/2016/research-in-teams/16rit690). The hospitality and support of BIRS is acknowledged with gratitude. Part of the research was carried out when Sergio Estrada visited Georgia Southern University, he wishes to thank the Department of Mathematical Sciences at Georgia Southern University for the support during his stay in May 2017}

\begin{abstract}
We prove that, for any $n \geq 2$, the classes of $\FP{n}$-injective modules and of $\FP{n}$-flat modules are both covering and preenveloping over any ring $R$. This includes the case of $FP_{\infty}$-injective and $FP_{\infty}$-flat modules (i.e. absolutely clean and, respectively, level modules).
Then we consider a generalization of the class of (strongly) Gorenstein flat modules - the (strongly) Gorenstein AC-flat modules (cycles of exact complexes of flat modules that remain exact when tensored with any absolutely clean module). We prove that some of the properties of  Gorenstein flat modules extend to the class of Gorenstein AC-flat modules; for example we show that this class is precovering over any ring $R$.  We also show that (as in the case of Gorenstein flat modules) every Gorenstein AC-flat module is a direct summand of a strongly Gorenstein AC-flat module. When $R$ is such that the class of Gorenstein AC-flat modules is closed under extensions, the converse is also true. We also prove that if the class of Gorenstein AC-flat modules is closed under extensions, then this class of modules is covering.
\end{abstract}

\maketitle

\section{Introduction}

It is well known that finitely generated and finitely presented modules (over a ring $R$) play an important part in homological algebra.

In \cite{bravo:15:fp} the authors  study a refinement of the class of finitely presented modules called \emph{finitely $n$-presented modules}; we will denote this class by $\FP{n}$. Following the characterization of noetherian and coherent rings in terms of these finitely $n$-presented modules,  they  generalize the notions of noetherian and coherent rings to what they call \emph{n-coherent rings}; see Section \ref{Preliminaries} for details. The $\FP{n}$-modules are also used to define the \emph{$\FP{n}$-injective } and the \emph{$\FP{n}$-flat} modules. These are the modules that have a vanishing property with respect to the functors $\Ext^1_R(M,-)$ and $\Tor^R _1 (M,-)$, whenever $M \in \FP{n}$. Thus they are generalizations of the injective and flat modules.

In the first part of this paper, we continue the study of these classes of modules. We show that for any $n \geq 2$, the class of $\FP{n}$-injective modules is both \emph{preenveloping} and \emph{covering} over an arbitrary ring $R$, thus improving  \cite[Theorem 4.4]{bravo:15:fp}. We also prove a similar result for the class of $\FP{n}$-flat modules, namely, that this class is both covering and {preenveloping} over any ring $R$; this improves \cite[Theorem 4.5]{bravo:15:fp} by showing the preenveloping property. Our results also include the case of $FP_{\infty}$-flat (i.e. the \emph{level} modules introduced in \cite{gillespie:14:stable}) and that of $FP_{\infty}$-injective modules (i.e. the \emph{absolute clean} modules introduced also in \cite{gillespie:14:stable}). We also give a sufficient condition for the existence of $\FPinj{n}$-envelopes. Also, a necessary and sufficient condition is given so that every module have an epimorphic $\FPflat{n}$-envelope.

We use {duality pairs} in order to prove some of the results mentioned above. This notion was introduced by Holm and J{\o}rgensen in \cite{holm:09:duality}.
Recall that a pair of classes of modules $(\mathcal{M}, \mathcal{C})$ is a \emph{duality pair} if it satisfies the following conditions:
\begin{enumerate}
\item $M \in \mathcal{M}$ if and only if $M^+ \in \mathcal{C}$.
\item $\mathcal{C}$ is closed under direct summands and under finite direct sums.
\end{enumerate}
As usual the Pontryagin dual of a left (respectively right) $R$-module  $M$ is defined as the right (respectively left) $R$-module $M^+ = \Hom_{\Z}(M, \Q/\Z)$. So in particular, if $\mathcal{M}$ is a class of left modules, then  $\mathcal{C}$ is a class of right modules. Holm and J{\o}rgensen also gave in  \cite{holm:09:duality} sufficient conditions for the class $\mathcal{M}$ be preenveloping and also to be precovering. These results proved to be very useful in studying the existence of precovers and preenvelopes with respect to various classes of modules; see for example \cite{iacob:13:gorinj}, \cite{holm:09:duality} and \cite{iacob:15:gor.inj}.

Following the ideas of preenveloping and precovering classes, we then consider some generalizations of the Gorenstein injective and Gorenstein flat modules.  We begin by noting that the class of Ding injective modules generalizes that of the Gorenstein injective modules; respectively the class of Ding projective modules  generalizes that of the Gorenstein projective modules. These classes  were  introduced by Ding and co-authors and called strongly Gorenstein FP-injective modules  and strongly Gorenstein flat modules, respectively.  The two classes of modules have some nice properties over coherent rings, analogous to the properties of the Gorenstein projectives and Gorenstein injectives over noetherian rings. Later Gillespie renamed strongly Gorenstein flat modules as Ding projective, and Gorenstein FP-injective as Ding injective; see \cite{gillespie:10:model} for details. We prove that the class of Ding injective modules is {enveloping} over any Ding-Chen ring. Thus we improve on a result of Gang, Liu, Liang, who proved in \cite[Corollary 3.3]{yang:13:cotorsion}  that this class is \emph{special preenveloping} over any Ding-Chen ring. This was also observed in \cite{gillespie:14:stable}; see the paragraph right before the proof of  \cite[Theorem 5.5]{gillespie:14:stable}.


We also give necessary and sufficient conditions for the class of Ding injective modules to be covering (Theorem \ref{theorem.Ding duality pair}).

 In the second part of the paper we consider a generalization of the class of Gorenstein flat modules, namely that of \emph{Gorenstein AC-flat} modules.   In  \cite{gillespie:14:stable}, the authors extend Gorenstein homological algebra to arbitrary rings. To do this, they replace the finitely generated (respectively finitely presented) modules with the modules of\emph{ type $FP_{\infty}$}; that is modules that have a resolution of finitely generated projective (or free) modules. In doing so, the injective (respectively FP-injective respectively) modules are replaced with the absolutely clean modules, and the flat modules are replaced by the level modules. In turn, the Gorenstein injective, and respectively Gorenstein projective, modules are replaced with the
 \emph{Gorenstein AC-injective}, and respectively \emph{Gorenstein AC-projective}, modules. We define the Gorenstein AC-flat modules as the cycles of exact complexes of flat modules that remain exact when tensoring with any absolutely clean module. We show that the class of Gorenstein AC-flat modules is precovering over any ring $R$
 and if the ring $R$ is such that the class of Gorenstein AC-flat modules is closed under extensions then this class is in fact covering.

We also define the notion of \emph{strongly Gorenstein AC-flat module} and we show that any Gorenstein AC-flat module is a direct summand of a strongly Gorenstein AC-flat module. When $R$ is such that the class of Gorenstein AC-flat modules is closed under extensions, the converse is also true.

\section{Preliminaries} \label{Preliminaries}

Let $R$ be a (unital associative) ring and $n$ a nonnegative integer. From now on, and unless otherwise stated, by $R$-module we mean a left $R$-module.

A module $M$ is said to be \emph{finitely $n$-presented} if there exists an exact sequence
\[
F_n \rightarrow F_{n-1} \rightarrow \ldots \rightarrow F_1 \rightarrow F_0 \rightarrow M \rightarrow 0,
\]
 with each $F_i$ a finitely generated projective (or free) module. Denote by $\FP{n}$ the class of all finitely $n$-presented modules. We say that a module $M$ is in $\FP{\infty}$ if and only if $M \in \FP{n}$ for all $n \ge 0$.

It is immediate from the definition that we have the following descending chain of inclusions:
\[
\text{FP}_0 \supseteq \text{FP}_1 \supseteq \ldots \supseteq \FP{n} \supseteq \text{FP}_{n+1} \supseteq \ldots \supseteq \FP{\infty}
\]
These classes of modules were studied in \cite{bravo:15:fp} and \cite{gillespie:14:stable}. We recall some of their properties.

\begin{proposition}[{\cite[Proposition 1.7,  Theorem 1.8]{bravo:15:fp}}] \label{prop-FPn}
Let $n \ge 1$ be an integer. Then:
\begin{enumerate}[leftmargin=*]
\item The class $\FP{n}$ is closed under direct summands, extensions, and cokernels of monomorphisms.
\item $\FP{\infty}$ is closed under kernels of epimorphisms.
\end{enumerate}
\end{proposition}
Recall that a class $W$ is thick if it is closed
under direct summands, and whenever we are given a short exact sequence
$0 \rightarrow A \rightarrow B \rightarrow C \rightarrow 0$ with any two out of the three terms  in $W$, then so is the third. We note that $\FP{\infty}$ is always thick.

By Proposition \eqref{prop-FPn}, the class $\FP{n}$ is thick if, and only if, it is closed under kernels of epimorphisms. Indeed, we recall the following result.
\begin{proposition}[{\cite[Theorem 2.4]{bravo:15:fp}}]
 The class $\FP{n}$ is thick, if and only if, $\FP{k} = \FP{\infty}$, for any $k \ge n$.
\end{proposition}
A ring $R$ satisfying the equivalent conditions above is said to be $n$-coherent. We note that the 0-coherent rings are simply the noetherian rings, the 1-coherent rings are the usual coherent rings, any ring is $\infty$-coherent. Also, we have the following chain of inclusions:
\[
0\text{-coherent} \subseteq 1\text{-coherent} \subseteq \cdots \subseteq n\text{-coherent} \subseteq (n+1)\text{-coherent} \subseteq \cdots \subseteq \infty\text{-coherent}.
\]

The $\FP{n}$-injective and $\FP{n}$-flat modules are defined in terms of the finitely $n$-presented modules. We recall the definitions next.

\begin{definition}[{\cite[ Definition 3.1, Definition 3.2]{bravo:15:fp}}]  \label{def-FPn-inj}
Let $n \geq 0$ (including the case $n=\infty$).
\begin{enumerate}[leftmargin=*]
\item A left R-module $M$ is $\FP{n}$-injective if $\Ext_R^1(F, M) =0$ for all $F \in \FP{n}$. We denote by $\FPinj{n}$ the class of all $\FP{n}$-injective modules.
\item A left R-module $M$ is $\FP{n}$-flat if $\Tor^R_1(F, M) =0$ for all $F \in \FP{n}$. We denote by $\FPflat{n}$ the class of all $\FP{n}$-flat modules.
\end{enumerate}
\end{definition}

With this definition, $M$ is injective if, and only if, $M$ is $\FP{0}$-injective, and $M$ is
FP-injective (in the sense of Stenstr\"{o}m \cite{stenstrom:70:coherent}) if, and only if, $M$ is $\FP{1}$-injective. The case of $\FP{\infty}$-injective modules
(i.e. absolutely clean) is the same as introduced in \cite{gillespie:14:stable}; that is those $M \in \Rmod$
for which $\Ext_R^1 (F, M) = 0$ for all $F \in FP_{\infty}$. We also note that this definition of $\FP{n}$-injective modules differs from that of J. Chen and N. Ding (\cite{ding:96:coherent}) for $n > 1 $ (they consider orthogonality with respect to $\Ext_R^n(-,-)$ instead).

From Definition \eqref{def-FPn-inj} we get the following ascending chain:
\[
\FPinj{0} \subseteq \FPinj{1} \subseteq \cdots \subseteq \FPinj{n}  \subseteq \FPinj{n+1} \subseteq \cdots  \subset \FPinj{\infty}.
\]
The $\FP{n}$-flat modules have a similar ascending chain and the class of flat modules also coincides with $\FPflat{0}$. Furthermore, since any module is a direct limit of finitely presented modules and since the $\Tor(-,-)$ functor commutes with direct limits, we note that the classes $\FPflat{0}$ and $\FPflat{1}$ coincide.
Note that the \emph{level modules}, in
the sense of \cite{gillespie:14:stable}, coincide with the $\FPflat{\infty}$ modules.

For completion, we summarize the following results about the classes of $\FP{n}$-injective modules and $\FP{n}$-flat modules presented in \cite{bravo:15:fp}. The first one has to do with the Pontryagin dual.
\begin{theorem*}[{\cite[Proposition 3.5 and Proposition 3.6]{bravo:15:fp}}]
Let $n>1$ (this includes the case $n = \infty$).
\begin{enumerate}[leftmargin=*]
\item $M$ is $\FP{n}$-injective if and only if $M^+$ is $\FP{n}$-flat.
\item $N$ is $\FP{n} $-flat if and only if $N^+$ is $\FP{n}$-injective.
\end{enumerate}
\end{theorem*}
The next one captures some closure properties of these classes.
\begin{theorem*}[{\cite[Proposition 3.9, Proposition 3.10 and Proposition 3.11]{bravo:15:fp}}]
Let $n>1$ (this includes the case $n = \infty$). Then the classes of $\FP{n}$-injective modules and $\FP{n}$-flat modules are closed under pure submodules, pure quotients, extensions, direct limits, products and direct summands.
\end{theorem*}

\section{Applications of duality pairs to $\FP{n}$-injective/flat covers and preenvelopes and Ding-injective covers and envelopes} \label{Applications}

 We will use duality pairs to prove the covering and preenveloping results mentioned in the introduction.
By  Holm and J{\o}rgensen, \cite[Definition 2.1]{holm:09:duality}, a duality pair $(\mathcal{M}, \mathcal{C})$ is \emph{coproduct closed} if $\mathcal{M}$ is closed under arbitrary direct sums. The duality pair $(\mathcal{M}, \mathcal{C})$ is \emph{product closed} if $\mathcal{M}$ is closed under direct products. A duality pair $(\mathcal{M}, \mathcal{C})$ is \emph{perfect} if $\mathcal{M}$ is closed under arbitrary direct sums, under extensions, and if $R \in \mathcal{M}$.
We also recall the following result.
\begin{theorem*}[{\cite[Theorem 3.1]{holm:09:duality}}]
Let $(\mathcal{M}, \mathcal{C})$ be a duality pair. Then $\mathcal{M}$ is closed under pure submodules, pure quotients and pure extensions. Also
\begin{enumerate}[leftmargin=*]
\item If $(\mathcal{M}, \mathcal{C})$ is product closed, then $\mathcal{M}$ is preenveloping.
\item If $(\mathcal{M}, \mathcal{C})$ is coproduct closed, then $\mathcal{M}$ is covering.
\item If $(\mathcal{M}, \mathcal{C})$ is perfect, then $(\mathcal{M}, \mathcal{M}^\bot)$ is a perfect cotorsion pair.
\end{enumerate}
\end{theorem*}
A pair of $R$-modules $(\mathcal{A},\mathcal{B})$ is said to be a \emph{cotorsion pair} if
\[
\mathcal{A} = {}^{\bot}\mathcal{B} = \{ X \in \Rmod : \Ext_R^1(X,B) \text{ for all } B \in \mathcal{B} \}
\]
and
\[
\mathcal{B} = \mathcal{A}^{\bot} = \{ X \in \Rmod : \Ext_R^1(A,X) \text{ for all } A \in \mathcal{A} \}.
\]
A cotorsion pair $(\mathcal{A},\mathcal{B})$ is said to be \emph{perfect} if $\mathcal{A}$ is covering and $\mathcal{B}$ is enveloping. Standard references for cotorsion pairs are \cite{E-J-Vol-1} and \cite{Gobel-Trlifaj}.

We will also use the following results from \cite{holm:08:purity}.
\begin{proposition}{\cite[Remark 2.6 and Theorem 4.3]{holm:08:purity}} \label{Results-from-HJ}
\begin{enumerate}[leftmargin=*]
\item Let $\mathcal{G}$ be a class of left $R$-modules that is closed under pure submodules. Then $\mathcal{G}$
is preenveloping if and only if it is closed under set indexed direct products.
\item Let $\mathcal{B}$ be a class of right $R$-modules. The class $\kker{\Tor^R_1(\mathcal{B}, -)}$ is the first half of a perfect cotorsion pair. \label{Ker-first-half}
\end{enumerate}
\end{proposition}

An immediate application is the following result.

\begin{proposition}[{\cite[Proposition 4.5]{bravo:15:fp}} and {\cite[Proposition 3.12]{GP}}]
For any ring $R$ and every $n \ge 0$, the pair $(\FPflat{n}, \FPflat{n}^{\bot})$ is a perfect cotorsion pair.
\end{proposition}

\begin{proof}
This follows from $\FPflat{n} = \kker{\Tor^R_1 (\FP{n},-)}$ and part \eqref{Ker-first-half} of Proposition \ref{Results-from-HJ}.
\end{proof}

In the following two results, we prove that, for any $n > 1$, the classes of $\FP{n}$-injective and $\FP{n}$-flat modules are both preenveloping and covering, over any ring $R$.
\begin{proposition}\label{prop.cov.env} ({\cite[Theorem 3.4]{GP}})
Let $n > 1$ be an integer. The class of $\FP{n}$-flat modules is both covering and preenveloping over any ring $R$.
\end{proposition}

\begin{proof}
Since $\FPflat{n}$ is the left half of a perfect cotorsion pair, it is a covering class.\\
By \cite[Proposition 3.5]{bravo:15:fp}, a module $M$ is $\FP{n}$-flat if and only if $M^+$ is $\FP{n}$-injective. Also, by \cite[Proposition 3.9 and Proposition 3.11]{bravo:15:fp},  the class of $\FP{n}$-flat modules is closed under direct sums and under summands. So the pair $(\FPflat{n}, \FPinj{n})$ is a duality pair.
Since, by \cite[Proposition 3.11]{bravo:15:fp} again, $\FP{n}$-flat is closed under direct products we obtain that $\FP{n}$-flat is also preenveloping.
\end{proof}

\begin{proposition}\label{prop.cov.env2}({\cite[Theorem 3.7]{GP}})
Let $n > 1$ be an integer. The class of $\FP{n}$-injective modules is both covering and preenveloping  over any ring $R$.
\end{proposition}

\begin{proof}
Since the class of $\FP{n}$-injective modules is closed under pure submodules and under set indexed direct products it is preenveloping (by Proposition 3.1. (1))
By {\cite[Propositions 3.6, Proposition 3.9 and Proposition 3.10]{bravo:15:fp}}, the pair $(\FPinj{n}, \FPflat{n})$ is a duality pair. Also, from \cite[Proposition 3.10]{bravo:15:fp}, the class $\FPinj{n}$ is closed under arbitrary direct sums, so the duality pair is coproduct closed. Therefore the class $\FPinj{n}$ is covering. 
\end{proof}




When the class of $\FP{n}$-injective modules also contains the ground ring $R$, another application of {\cite[Theorem 3.1]{holm:09:duality}} together with {\cite[Propositions 3.6, Proposition 3.9 and Proposition 3.10]{bravo:15:fp}} gives the following result.
\begin{proposition}[{\cite[Theorem 4.4]{bravo:15:fp}} and {\cite[Proposition 3.11]{GP}}]
If $n>1$ and $R$ is a self $\FP{n}$-injective ring then the pair $(\FPinj{n},\FPinj{n}^{\bot})$ is a perfect cotorsion pair.
\end{proposition}

A sufficient condition for the class of $\FPinj{n}$-modules be enveloping is that its left orthogonal class, $^{\bot} \FPinj{n}$, be closed under pure quotients. In terms of duality pairs, this condition is equivalent to $(^{\bot} \FPinj{n}, \FPflat{n}^{\bot})$ being a duality pair.
%
More precisely, we have:
\begin{proposition}
Let $n>1$ or $n=\infty$. The following statements are equivalent.
\begin{enumerate}[leftmargin=*]
\item The class ${}^{\bot} \FPinj{n}$ is closed under pure quotients. \label{left-perp-closed-pure}
\item The pair $({}^{\bot} \FPinj{n}, \FPflat{n}^{\bot})$ is a duality pair. \label{perps-duality-pair}
\end{enumerate}
If these conditions hold then the class of $\FPinj{n}$-modules is enveloping.
\end{proposition}

\begin{proof}
\eqref{left-perp-closed-pure} $\Rightarrow$ \eqref{perps-duality-pair}. We show first that $K \in {}^{\bot} \FPinj{n}$ if and only if $K^+ \in \FPflat{n}^{\bot}$. 

If $K \in {}^{\bot} \FPinj{n}$ then for any $G \in \FPflat{n}$ we have 
\[\Ext_R^1(G, K^+) \cong \Ext_R^1(K,G^+) = 0
\]
since $G^+ \in \FPinj{n}$. So $K^+ \in \FPflat{n}^{\bot}$ whenever $K \in {}^{\bot} \FPinj{n}$. 

Conversely, let $K$ be such that $K^+ \in \FPflat{n}^{\bot}$. By \cite[Corollary 4.2]{bravo:15:fp}, the pair $({}^{\bot} \FPinj{n}, \FPinj{n})$ is a complete cotorsion pair in $\Rmod$ for every $n \ge 0$. So there exists an exact sequence $0 \rightarrow A \rightarrow B \rightarrow K \rightarrow 0$ with $A \in \FPinj{n}$ and with $B \in {}^{\bot} \FPinj{n}$. This gives an exact sequence 
\[
0 \rightarrow K^+ \rightarrow B^+ \rightarrow A^+ \rightarrow 0 
\]
with $A^+ \in \FPflat{n}$ and with $B^+ \in \FPflat{n}^{\bot}$. Thus the sequence $0 \rightarrow K^+ \rightarrow B^+ \rightarrow A^+ \rightarrow 0$ splits. 
Thus the initial sequence, $0 \rightarrow A \rightarrow B \rightarrow K \rightarrow 0$, is pure exact. Since the class ${}^{\bot} \FPinj{n}$ is closed under pure quotients, it follows that $K \in {}^{\bot} \FPinj{n}$.

The class $^{\bot} \FPinj{n}$ is closed under direct summands and (arbitrary) direct sums, so $(^{\bot} \FPinj{n}, \FPflat{n}^{\bot})$ is a duality pair.\\
\eqref{perps-duality-pair} $\Rightarrow$ \eqref{left-perp-closed-pure} follows from the definition of a duality pair.

Finally, assume that condition \eqref{perps-duality-pair} holds. Then by \cite[Theorem 3.1]{holm:09:duality} the pair $({}^{\bot} \FPinj{n}, \FPinj{n})$ is a perfect cotorsion pair, so the class $\FPinj{n}$ is enveloping.
\end{proof}

%
%
%
%
%

We recall that, by \cite[Corollary 4.2]{bravo:15:fp},  the pair $(\FPflat{n}, \FPflat{n}^{\bot})$ is a perfect cotorsion pair in $\Rmod$ for every $n \ge 0$. Hence by \cite[Theorem 3.4]{holm:08:purity}, the class $\FPflat{n}^{\bot}$ is enveloping. A natural question is then to ask when is $\FPflat{n}$ enveloping. To answer this we must introduce the notion of $n$-hereditary rings. A ring is said to be left $n$-hereditary if every finitely $(n-1)$-presented submodule of a finitely generated projective left module is also a projective module. We refer the reader to \cite{BravoParra} for more details on $n$-hereditary rings; among the results available in that reference is the following.
\begin{proposition}{\cite[Proposition 23]{BravoParra}}
For any $n>1$, we have that  ring $R$ is $n$-hereditary if and only if submodules of any FP$_n$-Flat module is again an FP$_n$-Flat module.
\end{proposition}

With this at hand we can provide the following answer.

\begin{proposition}
Let $n>1$, then every module has an epimorphic $\FPflat{n}$-envelope if and only if $R$ is $n$-hereditary.
\end{proposition}
\begin{proof}
 We first recall that a class of $R$-modules is a pretorsion-free class if it is closed under submodules and direct products \cite[Chapter VI, \S 1]{rings-of-quot}.
 
 Now, R is $n$-hereditary if and only if $\FPflat{n}$ is closed under submodules, that is, if and only if $\FPflat{n}$ is a pretorsion-free class, given that $\FPflat{n}$ is always closed under direct products. From \cite[Proposition 4.1]{RadaSaorin}, we have that    $\FPflat{n}$ is a pretorsion-free class if and only if every left $R$-module has an epimorphic $\FPflat{n}$-envelope.
\end{proof}

Next, we recall that the FP-injective modules (in the sense of Stenstr\"{o}m \cite{stenstrom:70:coherent}) are the $\FP{1}$-injective modules in the sense of Bravo and P\'erez. A module $E$ is called FP-injective if $\Ext^1(A, E) = 0$ for all finitely presented $R$-modules $A$. The \emph{FP-injective dimension} of an $R$-module $B$ is defined to be the least integer $n \ge 0$ such that $\Ext^{n+1}(A, B) = 0$ for all finitely presented $R$-modules $A$. The FP-injective dimension of $B$ is denoted by $\FPid{B}$ and equals $\infty$ if no such $n$  exists.

We recall that $R$ is an \emph{$n$-FC ring} if it is both left and right coherent and $\FPid{_RR}$ and $\FPid{R_R}$ are both less than or equal to $n$. Ding and Chen introduced the {$n$-FC rings} in \cite{ding-chen-1} and \cite{ding:96:coherent}. Later, Gillespie in \cite{gillespie:10:model}, introduced the \emph{Ding-Chen rings}. A ring $R$ is called a Ding-Chen ring if it is an $n$-FC ring for some $n \ge 0$. Examples of Ding-Chen rings include all Gorenstein rings (in the sense of Iwanaga, i.e. left and right noetherian rings of finite self injective on both sides less or equal than $n$).

 We also recall that, over any ring $R$, an R-module $M$ is \emph{Ding injective} if there exists an exact
complex of injectives
\[
\textbf{I} =\cdots \rightarrow I_1 \rightarrow I_0 \rightarrow I_{-1} \rightarrow \cdots
\]
 with $M = \kker{I_0 \rightarrow I_{-1}}$
and such that  $\Hom_R(A, \textbf{I})$ is exact for any FP-injective module $A$. It follows from the definition that any Ding injective module is a Gorenstein injective module.  

 The dual notion is that of Ding projective modules. A module $G$ is \emph{Ding projective} if there exists an exact complex of projective modules
 \[
\textbf{P} = \cdots \rightarrow P_1 \rightarrow P_0 \rightarrow P_{-1} \rightarrow \cdots
 \]
with $G =\kker {P_0 \rightarrow P_{-1}}$ and such that  $\Hom_R(\textbf{P},F)$ is exact for any flat module $F$. In particular, any Ding projective module is Gorenstein projective.

For the rest of this section we let $\mathcal{W}$ denote the class of modules of finite flat dimension.  It is shown in \cite{ding:11:cotorsion}, that whenever $R$ is a Ding-Chen ring, then the pair  $(\mathcal{W}, \mathcal{W}^{\bot})$ is a complete hereditary cotorsion pair and in \cite{gillespie:10:model}, it is shown that the class $\mathcal{W}^\bot$ is the class of Ding-injective modules, which we denote by $\mathcal{DI}$.


We prove that over a Ding-Chen ring the complete cotorsion pair $(\mathcal{W}, \mathcal{DI})$ is actually a perfect cotorsion pair, and  therefore that the class $\mathcal{DI}$ is enveloping. Before doing this we need to recall a few definitions and results.

A module $N$ is said to be \emph{Gorenstein flat} if there exists an exact complex of flat modules
\[
\textbf{F} = \cdots \rightarrow F_1 \rightarrow F_0 \rightarrow F_{-1} \rightarrow \cdots
\]
such that $I \otimes \textbf{F}$ is still exact for any injective module $I$ and  $N$ is one of the cycles of the complex $\textbf{F}$. We denote by $\mathcal{GF}$ the class of Gorenstein flat modules. Let $\mathcal{GC}$ denote the class of \emph{Gorenstein cotorsion} modules, that is the right $\Ext^1$-orthogonal class of that of Gorenstein flat modules. We know from \cite[Theorem 2.12]{Enochs-Jenda-Lopez_Ramos} that over any coherent ring $R$, the pair $(\mathcal{GF}, \mathcal{GC})$ is a complete hereditary cotorsion pair.

We also recall the following results.

\begin{theorem*}[{\cite[Theorem 1]{Cheatham-Stone}}] The following statements are equivalent:
\begin{enumerate}[leftmargin=*]
\item $R$ is a left coherent ring.
\item $_RM$ is absolutely pure if and only if $M^{ + +}$ is an injective left $R$-module.
\item $M_R$ is flat if and only if $M^{ + +}$ is a flat right $R$-module.
\end{enumerate}
\end{theorem*}

\begin{theorem*}[{\cite[Theorem 4.2]{yang:13:cotorsion}}] Let $R$ be a Ding-Chen ring and $M$ an $R$-module. Then $M$ has finite flat dimension if and only if $M$ has finite FP-injective (or absolutely pure) dimension
\end{theorem*}

With all this at hand, we have the following result.

\begin{proposition} \label{(W,GC)-dual-pair}
Let $R$ be a Ding-Chen ring. Then $(\mathcal{W}, \mathcal{GC})$ is a perfect duality pair.
\end{proposition}

\begin{proof}
We check first that $K \in \mathcal{W} = {}^{\bot}\! \mathcal{DI}$ if and only if $K^+ \in \mathcal{GC}$.

Let $K \in \mathcal{W}$ and  $G \in \mathcal{GF}$. Then $\Ext^1 (G, K^+) \cong \Ext^1(K, G^+) =0$, by a double application of  \cite[Lemma 1.2.11]{Gobel-Trlifaj} and since $G^+$ is a Ding-injective module; this last statement is due to \cite[Lemma 2.8]{ding:08:cotorsion}.

For the converse, let $K$ be an $R$-module such that $K^+$ is a Gorenstein cotorsion module. Since $(\mathcal{W}, \mathcal{DI})$  is a complete cotorsion pair, there is a short exact sequence
\[
\textbf{E} = 0 \rightarrow A \rightarrow B \rightarrow K \rightarrow 0,
\]
with $A \in \mathcal{DI}$ and $B \in \mathcal{W}$. Then we have an exact sequence $0 \rightarrow K^+ \rightarrow B^+ \rightarrow A^+ \rightarrow 0$, with $K^+$ Gorenstein cotorsion, $B^+$  also Goresntein cotorsion (from the first part of the proof) and with $A^+$ Gorenstein flat (this follows from \cite[Proposition 12]{ding:96:coherent}, since any Ding injective module is Gorenstein injective). Since $\Ext^1(\mathcal{GF},\mathcal{GC})=0$, then the previous sequence splits and so $B^+ \cong A^+ \oplus K^+$. Thus $A^+$ is both Gorenstein cotorsion and Gorenstein flat.
Since $A^+ \in \mathcal{GF}$, it fits in a short exact sequence
\[
0 \to A^+ \to M \to N \to 0,
\]
with $M$ a flat module and $N \in \mathcal{GF}$. However, since $A^+ \in \mathcal{GC}$ also, any such short exact sequence splits. It follows that $A^+$ is flat, and so $A^{++}$ is injective. Since $R$ is coherent, then we have that $A$ is absolutely pure (FP-injective). Since the ring is Ding-Chen, then by \cite[Theorem 4.2]{gillespie:10:model}, the module $A$ has finite flat dimension. Therefore, the short exact sequence $\textbf{E}$ has $A, B \in \mathcal{W}$, which gives that $K \in \mathcal{W}$.


Since $\mathcal{W} = {}^{\bot} \mathcal{DI}$, then this class is closed under direct sums, summands, and also $R \in \mathcal{W}$. Therefore  $(\mathcal{W}, \mathcal{GC})$ is a perfect duality pair.
\end{proof}

\begin{corollary}
The class of Ding injective modules is enveloping over any Ding-Chen ring.
\end{corollary}

\begin{proof}
Since $(\mathcal{W}, \mathcal{GC})$ is a perfect duality pair, it follows from \cite{holm:09:duality}, that $(\mathcal{W}, \mathcal{DI})$ is a perfect cotorsion pair. So $\mathcal{DI}$ is enveloping.
\end{proof}

It was recently showed that over a Ding-Chen ring the class of Ding injective modules coincides with, $\mathcal{GI}$, the class of Gorenstein injective modules. For completion, we include  here the result along with a proof of this fact.

\begin{proposition} \label{GI=DI}
Let $R$ be a Ding-Chen ring. Then $\mathcal{GI} = \mathcal{DI}$.
\end{proposition}%
\begin{proof}
Since any Ding injective module is Gorenstein injective, we show the other inclusion. Let $D$ be a Gorenstein injective module.  Then $D = Z_0(\textbf{E})$, for some exact sequence of injective modules, $\textbf{E} = \cdots \to E_1 \to E_0 \to E_{-1} \to \cdots$, such that $\Hom(F,\textbf{E})$ is exact for all injective modules $F$. By \cite[Corollary 5.9]{Sto14}, we have $\Ext^1(F,D)=0$ for all modules $F$ of finite flat dimension, that is, for all $F\in \mathcal{W}$. Since $R$ is a Ding-Chen ring, we have $\mathcal{W}^{\bot}=\mathcal{DI}$. Therefore $D$ is Ding injective.
\end{proof}%

So we get the following  result.

\begin{proposition}
The class of Gorenstein injective modules is enveloping over any Ding-Chen ring.
\end{proposition}

We also consider the existence of the Gorenstein injective (pre)covers over Ding-Chen rings. We will use the following lemma.




\begin{lemma} \label{R-is-Noeth}
Let $R$ be a Ding-Chen ring. If the class of Gorenstein injective $R$-modules is precovering, then $R$ is left noetherian.
\end{lemma}

\begin{proof}
Since the class of Gorenstein injectives is precovering, it is closed under arbitrary direct sums by \cite[Proposition 1.2]{holm:08:purity}. From \cite[Corollary 3.5]{gillespie:10:model}, we get that $R$ is noetherian.
\end{proof}

Now we can prove the following.

\begin{proposition}
Let $R$ be a Ding-Chen ring. Then $R$ is a Gorenstein ring if and only if both the class of Gorenstein injective left $R$-modules and the class of Gorenstein injective right $R$-modules are precovering.
\end{proposition}

\begin{proof}
If $R$ is a Gorenstein ring, then \cite[Theorem 11.1.1]{E-J-Vol-1} gives that the class of Gorenstein injective modules is precovering.

For the converse, we get from Lemma \ref{R-is-Noeth}, that $R$ is a two sided noetherian ring. Since $R$ is also a Ding-Chen ring, then it has finite self injective dimension on both sides, and so it is a Gorenstein ring.
\end{proof}

A sufficient condition for the existence of the Gorenstein injective covers is the following.

\begin{proposition} \label{prop-GI-covering}
Let $R$ be any ring. If the class of Gorenstein injectives is closed under pure submodules, then it is covering.
\end{proposition}

\begin{proof}
Let $0 \rightarrow A \rightarrow B \rightarrow C \rightarrow 0$ be a pure exact sequence with $B$ Gorenstein injective. Since this class is closed under pure submodules, it follows that $A$ is also Gorenstein injective. But the class of Gorenstein injectives is closed under cokernels of monomorphisms \cite[Theorem 10.1.4]{E-J-Vol-1}, so $C$ is also Gorenstein injective. Thus the class of Gorenstein injectives, $\mathcal{GI}$, is also closed under pure quotients in this case.

Next, since any direct sum is a pure submodule of a direct product and that $\mathcal{GI}$ is closed under direct products, then it follows that it is also closed under direct sums. Thus from \cite[Theorem 2.5]{holm:08:purity}, we see that $\mathcal{GI}$ is covering.
\end{proof}

Next, we prove that,  when $R$ is a Ding-Chen ring, the class $\mathcal{GI}$, of Gorenstein injectives modules, is covering if and only if it is closed under pure submodules. More precisely we have:

\begin{theorem}\label{theorem.Ding duality pair}
Let $R$ be a Ding-Chen ring. The following statements are equivalent:
\begin{enumerate}[leftmargin=*]
\item The class $\mathcal{GI}$ is closed under pure submodules. \label{GI-pure}

\item The class $\mathcal{GI}$  is covering. \label{GI-covering}

\item The pair $(\mathcal{GI}, \mathcal{GF})$ is a duality pair. \label{GI-GF-duality-pair}

\end{enumerate}
\end{theorem}

\begin{proof}
%
%
%
%
%

\eqref{GI-pure} $\Rightarrow$ \eqref{GI-covering}.
This is Proposition \ref{prop-GI-covering}.

\eqref{GI-covering} $\Rightarrow$ \eqref{GI-GF-duality-pair}.
Since $\mathcal{GI}$ is covering, Lemma \eqref{R-is-Noeth} gives that $R$ is left noetherian.
%
%
Now suppose that $K \in \mathcal{GI}$, then  by \cite[Proposition 12]{ding:96:coherent},  we get that $K^+ \in \mathcal{GF}$. Conversely, suppose  that $K^+$ is Gorenstein flat. Since the ring $R$ is left noetherian there exists an exact sequence
\begin{equation} \label{ses}
0 \rightarrow K \rightarrow G \rightarrow L \rightarrow 0.
\end{equation}
 with $G \in \mathcal{GI}$ and $L \in {}^\bot \mathcal{GI} = \mathcal{W}$. Therefore we have an exact sequence
 \[
0 \rightarrow L^+ \rightarrow G^+ \rightarrow K^+ \rightarrow 0,
\]
 with $G^+ \in \mathcal{GF}$ and  $L^+ \in \mathcal{GC}$, by Proposition \eqref{(W,GC)-dual-pair}.

Since $\Ext^1(K^+, L^+) = 0$, then $G^+ \cong L^+ \oplus K^+$,  and therefore $L^+ \in \mathcal{GF}$. It follows that $L^+ \in \mathcal{GF} \cap \mathcal{GC}$, so $L^+$ is a flat
  $R$-module, and from \cite[Corollary 3.2.17]{E-J-Vol-1}, we have that $L$ is an injective module.

Thus we get that the short exact sequence in $\eqref{ses}$ has $G \in \mathcal{GI}$ and $L$ an injective modules; this gives that $K$ has finite Gorenstein injective dimension \cite[Theorem 2.25]{holm:05:gor.dim}. From \cite[Lemma 2.18]{holm:05:gor.dim}, there exists an
exact sequence
\[
0 \rightarrow B \rightarrow H \rightarrow K \rightarrow 0,
\]
with $B \in \mathcal{GI}$ and $H$ such that $\text{id}_R H = \text{Gid}_R K < \infty$. Hence the exact sequence
\[
0 \rightarrow K^+ \rightarrow H^+ \rightarrow B^+ \rightarrow 0
\] has $B^+, K^+ \in \mathcal{GF}$; this in turn gives that $H^+ \in \mathcal{GF}$. Since $\text{id}_R H$ is finite, it follows from \cite[Theorem 3.2.19]{E-J-Vol-1} that $H^+$ has finite flat dimension. But a Gorenstein flat module of finite flat dimension is flat (\cite[Corollary 10.3.4]{E-J-Vol-1}).
So $H^+$ is flat and therefore $H$ is injective. Thus $\text{Gid}_R K = 0$.
We know from \cite[Proposition 3.2]{holm:05:gor.dim} that the class of Gorenstein flat modules is closed under (arbitrary) direct sums, and under direct summands. Hence, the pair $(\mathcal{GI}, \mathcal{GF})$ is a duality pair.

\eqref{GI-GF-duality-pair}  $\Rightarrow$ \eqref{GI-pure}. From the definition of duality pairs, we have that the class $\mathcal{GI}$ is closed under pure submodules.

\end{proof}

\begin{remark}
We recently learned that Zhao and P\'erez in \cite{ZP17} also proved the results in propositions \ref{prop.cov.env} and \ref{prop.cov.env2} in a recent preprint of theirs. We thank the referees for pointing out the results in Propositions 3.2, 3.3, 3.4 and 3.5 are also in \cite{GP} (to appear in the Bulletin of the Malaysian Math. Sci. Soc.) 
\end{remark}

\section{Gorenstein AC-flat modules}


As already mentioned in the introduction, in \cite{gillespie:14:stable}  Gorenstein homological algebra is extended to arbitrary rings. To accomplish this the authors replace the finitely generated (respectively finitely presented) modules with the modules of type $FP_{\infty}$. In doing so, the injective (respectively FP-injective) modules are replaced with the absolutely clean modules, and the flat modules are replaced by the level modules. In turn, the Gorenstein injective (respectively Gorenstein projective) modules are replaced with the Gorenstein AC-injective (respectively Gorenstein AC-projective). Following the ideas of \cite{gillespie:14:stable}, we define the class of Gorenstein AC-flat modules.

\begin{definition}
A module $M$ is Gorenstein AC-flat if there exists an exact complex of flat modules $\textbf{F} = \ldots \rightarrow F_1 \rightarrow F_0 \rightarrow F_{-1} \rightarrow \ldots$ such that $A \otimes F$ is still exact for any absolutely clean module $A$ and such that $M$ is one of the cycles of the complex $\textbf{F}$.
We denote by $\mathcal{GF}_{\text{ac}}$ the class of the Gorenstein AC-flat modules.
\end{definition}

Since every injective module is an absolutely clean module, it follows that any Gorenstein AC-flat module is also a Gorenstein flat module. We note that when the ring $R$ is coherent, the two classes of modules coincide, because in this case the absolutely clean modules coincide with the absolutely pure modules; see \cite{gillespie:14:stable}.

A stronger notion is that of \emph{strongly Gorenstein AC-flat} module.

\begin{definition}
A module $M$ is strongly Gorenstein AC-flat if there is an exact complex of flat modules $\textbf{F} = \ldots \rightarrow F \xrightarrow{d} F \xrightarrow{d} F \xrightarrow{d} \ldots$ such that $A \otimes F$ is still exact for any absolutely clean module $A$ and such that $M = \kker{F \rightarrow F}$. 
\end{definition}

Note that every Gorenstein AC-projective module is Gorenstein AC-flat. This follows from the definition. Recall that a module $M$ is Gorenstein AC-projective if there exists an exact complex of projective modules $\textbf{F} = \ldots \rightarrow F_1 \rightarrow F_0 \rightarrow F_{-1} \rightarrow \ldots$ such that $A \otimes \textbf{F}$ is still exact for any absolutely clean module $A$ and such that $M = Z_0(\textbf{F})$.

In the remainder of this section, we show that some of the properties of the Gorenstein flat modules extend to the class of Gorenstein AC-flat modules.
We note first that, as in the case of Gorenstein flat modules, the class of Gorenstein AC-flat modules is precovering over any ring $R$.

\begin{proposition}
The class of Gorenstein AC-flat modules is precovering over any ring $R$.
\end{proposition}

\begin{proof}
Let $\widetilde{\mathcal F}$ be the class  of exact complexes of flat modules $\textbf{F}$ such that $A\otimes \textbf{F}$ is exact, for each absolutely clean right $R$-module $A$. Then, by \cite[Theorem 3.7]{EG15}, the class $\widetilde{\mathcal F}$ is special precovering. Finally we can argue as in the proof of Theorem A in \cite{yang:14:gorflatprec} to infer that every module has a Gorenstein AC-flat precover.
\end{proof}

We show that (as in the case of Gorenstein flat modules) if the class of Gorenstein AC-flat modules is closed under extensions, then it is also covering.

The following result gives some equivalent characterizations of the Gorenstein AC-flat modules.

\begin{lemma} \label{charac-GF_ac}
Let $R$ be any ring. The following are equivalent:
\begin{enumerate}[leftmargin=*]
\item $M$ is a Gorenstein AC-flat left $R$-module. \label{M-Gor}
\item $M$ is a left $R$-module and satisfies the following two conditions: \label{two-cond}
\begin{enumerate}
   \item $Tor_i (A,M) = 0$ for all $i > 0$ and all absolutely clean right $R$-modules $A$. \label{cond-1}
  \item There exists an exact sequence $0 \rightarrow M \rightarrow F^0 \rightarrow F^1 \rightarrow \ldots$ with each $F^i$ a flat left $R$-module and such that $A \otimes -$ leaves it exact whenever $A$ is an absolutely clean right $R$-module.\label{cond-2}
\end{enumerate}
\item There exists a short exact sequence of left $R$-modules
\[
0 \rightarrow M \rightarrow F \rightarrow G \rightarrow 0
\]
with $F$ flat and $G$ Gorenstein AC-flat \label{ses-GacFlata}
\end{enumerate}

\end{lemma}

\begin{proof}
The argument in \cite[Lemma 2.4]{bennis:12:homological} gives the result, once injective is replaced by absolutely clean.
\end{proof}

\begin{proposition}
Let $R$ be a ring such that the class of Gorenstein AC-flat modules is closed under extensions. Then the class of Gorenstein AC-flat modules, $\mathcal{GF}_{ac}$, is closed under direct limits.
\end{proposition}

\begin{proof}
The argument in the proof of \cite[Lemma 3.1]{gang:12:gorflat} gives the result, once injective is replaced by absolutely clean.
\end{proof}

\begin{proposition}
Let $R$ be a ring such that the class $\mathcal{GF}_{ac}$ is closed under extensions. Then $\mathcal{GF}_{ac}$ is covering.
\end{proposition}

\begin{proof}
The class $\mathcal{GF}_{ac}$ is precovering and closed under direct limits, so in this case it is a covering class of modules.
\end{proof}

Recall that a class of modules $\mathcal{X}$  is said to be \emph{projectively resolving} is it is closed under kernels of epimorphism between modules from $\mathcal{X}$, extension and contains the class of projective modules.

\begin{lemma} \label{GF-ac-proj-resolv}
If the class of Gorenstein AC-flat modules is closed under extensions, then it is projectively resolving.
\end{lemma}

\begin{proof}
To claim that the class of Gorenstein AC-flat modules is projectively resolving, it
suffices to prove that it is closed under kernels of epimorphisms. Then, consider
a short exact sequence of left $R$-modules $0 \rightarrow A \rightarrow B \rightarrow C \rightarrow 0$, where $B$ and $C$ are Gorenstein AC-flat.
We prove that $A$ is Gorenstein AC-flat. Since $B$ is Gorenstein AC-flat, there exists a short exact sequence
of left $R$-modules $0 \rightarrow B \rightarrow F \rightarrow G \rightarrow 0$, where $F$ is flat and $G$ is Gorenstein AC-flat. Consider the
following pushout diagram:

\[
\begin{diagram}
\node{}\node{}\node{0}\arrow{s}\node{0}\arrow{s}\\
\node{0}\arrow{e}\node{A}\arrow{s,=}\arrow{e}\node{B}\arrow{s}\arrow{e}\node{C}\arrow{s}\arrow{e}\node{0}\\
\node{0}\arrow{e}\node{A}\arrow{e}\node{F}\arrow{s}\arrow{e}\node{D}\arrow{s}\arrow{e}\node{0}\\
\node{}\node{}\node{G}\arrow{s}\arrow{e,=}\node{G}\arrow{s}\\
\node{}\node{}\node{0}\node{0}
\end{diagram}
\]
From the right vertical sequence, and since $\mathcal{GF}_{\text{ac}}$ is closed under extensions, we have that the $R$-module $D$ is Gorenstein AC-flat.
Therefore, since $F$ is flat, we get that $A$ is Gorenstein AC-flat, by Lemma \eqref{charac-GF_ac}.
\end{proof}

We recall that if $\mathcal{X}$ is a class of modules that is projectively resolving and closed under countable direct sums, then $\mathcal{X}$ is also closed under direct summands; see  \cite[Proposition 1.4]{holm:05:gor.dim}.

\begin{proposition}
Let $R$ be a ring such that the class $\mathcal{GF}_{ac}$ is closed under extensions. Then $(\mathcal{GF}_{ac}, \mathcal{GF}_{ac}^\bot)$ is a complete and hereditary cotorsion pair.
\end{proposition}

\begin{proof}
We begin by checking that we have a cotorsion pair.

Let $X \in {}^\bot  (\mathcal{GF}_{ac}^\bot )$. Since $\mathcal{GF}_{ac}$ is covering there is an exact sequence $0 \rightarrow A \rightarrow B \rightarrow X \rightarrow 0$ with $B \in \mathcal{GF}_{ac}$ and  $A \in \mathcal{GF}_{ac}^\bot$. Then $\Ext^1(X, A) =0$, so we have $B \simeq A \oplus X$ and so $X \in \mathcal{GF}_{ac}$.

The cotorsion pair is complete because $\mathcal{GF}_{ac}$ is covering and any cover is a special precover, so there are enough projectives. From Lemma \eqref{GF-ac-proj-resolv}, we get that the class of Gorenstein AC-flat modules is closed under kernels of epimorphisms, and hence the cotorsion pair is also hereditary.
\end{proof}

We also prove an analogue result of \cite[Theorem 3.5]{bennis:07:strongly}.

\begin{theorem}  \label{charac-of-GF-ac-via-SGF-ac}
Let $R$ be any ring.
\begin{enumerate}[leftmargin=*]
\item If a module is Gorenstein AC-flat, then it is a direct summand of a strongly Gorenstein AC-flat module. \label{direct-summand-SGF-ac}
\item If the ring $R$ is such that the class of Gorenstein AC-flat modules is closed under extensions, then every direct summand of a Gorenstein AC-flat module is Gorenstein AC-flat.  \label{closed-extensions-gives-direct-summand}
\end{enumerate}

\end{theorem}

\begin{proof}
\eqref{direct-summand-SGF-ac} Let $M$ be a Gorenstein AC-flat module. Then there exists an exact complex of flat modules
\[
\textbf{F} = \; \cdots \rightarrow F_1 \rightarrow F_0 \rightarrow F_{-1} \rightarrow \cdots
\]
 such that $A \otimes \textbf{F}$ is exact for any absolutely clean $R$-module $A$ and such that $M = \kker { F_0 \rightarrow F_{-1}}$.  Consider the complex $\textbf{X} = \bigoplus_{n \in Z} \textbf{F}[n]$, where $\textbf{F}[n]$ is the complex $\textbf{F}$ shifted by $n$ with the corresponding differential.
 Then $\textbf{X}$ is still an exact complex, with the property that $A \otimes \textbf{X}$ is still exact for any absolutely clean $R$-module $A$, and $\textbf{X}$ is of the form
 \[
 \textbf{X} = \cdots \rightarrow T \rightarrow T \rightarrow T \rightarrow \cdots.
 \]
 Then $\kker {T \rightarrow T}$ is a strongly Gorenstein AC-flat module and $M$ is a direct summand of $\kker{T \rightarrow T}$.

\eqref{closed-extensions-gives-direct-summand} Since the class of Gorenstein AC-flat modules is closed under extensions, we have that $(\mathcal{GF}_{ac}, \mathcal{GF}_{ac}^\bot)$ is a complete hereditary cotorsion pair.

Let $G$ be a Gorenstein AC-flat module, and let $M$ be a direct summand of $G$. For any $K \in  \mathcal{GF}_{ac}^\bot$ we have $\Ext^1 (G,K)=0$. Since $\Ext^1(M,K)$ is a direct summand of $\Ext^1(G,K)$, it follows that $\Ext^1(M,K) = 0$ for all $K \in  \mathcal{GF}_{ac}^\bot$. Thus $M \in \mathcal{GF}_{ac}$.
\end{proof}

\begin{corollary}
Let $R$ be a ring such that the class of Gorenstein AC-flat modules is closed under extensions. Then a module is Gorenstein AC-flat if and only if it is a direct summand of a strongly Gorenstein AC-flat module.
\end{corollary}

\bibliographystyle{alpha}
\bibliography{Duality}

\begin{thebibliography}{EJLR04}

\bibitem[Ben09]{bennis:12:homological}
Driss Bennis.
\newblock Rings over which the class of {G}orenstein flat modules is closed
  under extensions.
\newblock {\em Comm. Algebra}, 37(3):855--868, 2009.



\bibitem[BM07]{bennis:07:strongly}
Driss Bennis and Najib Mahdou.
\newblock Strongly {G}orenstein projective, injective, and flat modules.
\newblock {\em J. Pure Appl. Algebra}, 210(2):437--445, 2007.

\bibitem[BGH14]{gillespie:14:stable}
Daniel Bravo, James {Gillespie}, and Mark {Hovey}.
\newblock {The stable module category of a general ring}.
\newblock {\em ArXiv e-prints}, May 2014.

\bibitem[BP17a]{BravoParra}
Daniel Bravo and Carlos~E. {Parra}.
\newblock {Torsion pairs over $n$-Hereditary rings}.
\newblock {\em ArXiv e-prints}, May 2017.

\bibitem[BP17b]{bravo:15:fp}
Daniel Bravo and Marco~A. P\'erez.
\newblock Finiteness conditions and cotorsion pairs.
\newblock {\em J. Pure Appl. Algebra}, 221(6):1249--1267, 2017.

\bibitem[CS81]{Cheatham-Stone}
Thomas~J. Cheatham and David~R. Stone.
\newblock Flat and projective character modules.
\newblock {\em Proc. Amer. Math. Soc.}, 81(2):175--177, 1981.

\bibitem[DC93]{ding-chen-1}
Nanqing Ding and Jian~Long Chen.
\newblock The flat dimensions of injective modules.
\newblock {\em Manuscripta Math.}, 78(2):165--177, 1993.

\bibitem[DC96]{ding:96:coherent}
Nanqing Ding and Jianlong Chen.
\newblock Coherent rings with finite self-{$FP$}-injective dimension.
\newblock {\em Comm. Algebra}, 24(9):2963--2980, 1996.

\bibitem[EG15]{EG15}
Sergio {Estrada} and James. {Gillespie}.
\newblock {The projective stable category of a coherent scheme}.
\newblock {\em To appear in Proc. Roy. Soc. Edinburgh Sect. A.}, November 2015.

\bibitem[EI15]{iacob:13:gorinj}
Edgar~E. Enochs and Alina Iacob.
\newblock Gorenstein injective covers and envelopes over {N}oetherian rings.
\newblock {\em Proc. Amer. Math. Soc.}, 143(1):5--12, 2015.

\bibitem[EJ11]{E-J-Vol-1}
Edgar~E. Enochs and Overtoun M.~G. Jenda.
\newblock {\em Relative homological algebra. {V}olume 1}, volume~30 of {\em De
  Gruyter Expositions in Mathematics}.
\newblock Walter de Gruyter GmbH \& Co. KG, Berlin, extended edition, 2011.

\bibitem[EJLR04]{Enochs-Jenda-Lopez_Ramos}
Edgar~E. Enochs, Overtoun M.~G. Jenda, and J.~A. Lopez-Ramos.
\newblock The existence of {G}orenstein flat covers.
\newblock {\em Math. Scand.}, 94(1):46--62, 2004.

\bibitem[Gil10]{gillespie:10:model}
James Gillespie.
\newblock Model structures on modules over {D}ing-{C}hen rings.
\newblock {\em Homology Homotopy Appl.}, 12(1):61--73, 2010.

\bibitem[GP17]{GP}
Zenghui Gao and Jie Peng.
\newblock {F}{P}{$_n$}-{I}njective and {F}{P}{$_n$}-{F}lat complexes.
\newblock {\em Bulletin of the Malaysian Mathematical Sciences Society}, May
  2017.

\bibitem[GT06]{Gobel-Trlifaj}
R\"udiger G\"obel and Jan Trlifaj.
\newblock {\em Approximations and endomorphism algebras of modules}, volume~41
  of {\em De Gruyter Expositions in Mathematics}.
\newblock Walter de Gruyter GmbH \& Co. KG, Berlin, 2006.

\bibitem[GZ12]{gang:12:gorflat}
Yang Gang and Liu Zhongkui.
\newblock Gorenstein flat covers over gf-closed rings.
\newblock {\em Communications in Algebra}, 40(5):1632--1640, 2012.

\bibitem[HJ08]{holm:08:purity}
Henrik Holm and Peter J{\o}rgensen.
\newblock Covers, precovers, and purity.
\newblock {\em Illinois J. Math.}, 52(2):691--703, 2008.

\bibitem[HJ09]{holm:09:duality}
Henrik Holm and Peter J{\o}rgensen.
\newblock Cotorsion pairs induced by duality pairs.
\newblock {\em J. Commut. Algebra}, 1(4):621--633, 2009.

\bibitem[Hol04]{holm:05:gor.dim}
Henrik Holm.
\newblock Gorenstein homological dimensions.
\newblock {\em J. Pure Appl. Algebra}, 189(1-3):167--193, 2004.

\bibitem[Iac17]{iacob:15:gor.inj}
Alina Iacob.
\newblock Gorenstein injective envelopes and covers over two sided noetherian
  rings.
\newblock {\em Comm. Algebra}, 45(5):2238--2244, 2017.

\bibitem[MD07]{ding:11:cotorsion}
Lixin Mao and Nanqing Ding.
\newblock Envelopes and covers by modules of finite {FP}-injective and flat
  dimensions.
\newblock {\em Comm. Algebra}, 35(3):833--849, 2007.

\bibitem[MD08]{ding:08:cotorsion}
Lixin Mao and Nanqing Ding.
\newblock Gorenstein {FP}-injective and {G}orenstein flat modules.
\newblock {\em J. Algebra Appl.}, 7(4):491--506, 2008.

\bibitem[RS98]{RadaSaorin}
Juan Rada and Manuel Saorin.
\newblock Rings characterized by (pre)envelopes and (pre)covers of their
  modules.
\newblock {\em Comm. Algebra}, 26(3):899--912, 1998.

\bibitem[Ste70]{stenstrom:70:coherent}
Bo~Stenstr{\"o}m.
\newblock Coherent rings and fp-injective modules.
\newblock {\em Journal of the London Mathematical Society}, 2(2):323--329,
  1970.

\bibitem[Ste75]{rings-of-quot}
Bo~Stenstr{\"o}m.
\newblock {\em Rings of quotients}.
\newblock Springer-Verlag, New York-Heidelberg, 1975.
\newblock Die Grundlehren der Mathematischen Wissenschaften, Band 217, An
  introduction to methods of ring theory.

\bibitem[{\v{S}}{\v{t}}o14]{Sto14}
Jan {\v{S}}{\v{t}}ov{\'{\i}}{\v{c}}ek.
\newblock {On purity and applications to coderived and singularity categories}.
\newblock {\em ArXiv e-prints}, December 2014.

\bibitem[YL14]{yang:14:gorflatprec}
Gang Yang and Li~Liang.
\newblock All modules have {G}orenstein flat precovers.
\newblock {\em Comm. Algebra}, 42(7):3078--3085, 2014.

\bibitem[YLL13]{yang:13:cotorsion}
Gang Yang, Zhongkui Liu, and Li~Liang.
\newblock Ding projective and {D}ing injective modules.
\newblock {\em Algebra Colloq.}, 20(4):601--612, 2013.

\bibitem[ZP17]{ZP17}
Tiwei {Zhao} and Marco~A. {P{\'e}rez}.
\newblock {Relative FP-injective and FP-flat complexes and their model
  structures}.
\newblock {\em ArXiv e-prints}, March 2017.

\end{thebibliography}

\end{document}